\documentclass[12pt]{amsart}
\usepackage{amsfonts, amsbsy, amsmath, amssymb}

\newtheorem{thm}{Theorem}[section]
\newtheorem{lem}[thm]{Lemma}

\newtheorem{prop}[thm]{Proposition}

\newtheorem{rmk}[thm]{Remark}

\newtheorem{thm-con}[thm]{Theorem-Conjecture}
\numberwithin{equation}{section}

\theoremstyle{definition}

\newcommand{\f}{\Bbb F}

\begin{document}

\title[Reversed Dickson Polynomials of the $(k+1)$-th kind]{Reversed Dickson Polynomials of the $(k+1)$-th kind over finite fields}

\author[Neranga Fernando]{Neranga Fernando}
\address{Department of Mathematics,
Northeastern University, Boston, MA 02115}
\email{w.fernando@neu.edu}

\begin{abstract}
We discuss the properties and the permutation behaviour of the reversed Dickson polynomials of the $(k+1)$-th kind $D_{n,k}(1,x)$ over finite fields. The results in this paper unify and generalize several recently discovered results on reversed Dickson polynomials over finite fields. 
\end{abstract}



\maketitle

\section{Introduction}

Let $p$ be a prime and $q$ a power of $p$. Let $\Bbb F_q$ be the finite field with $q$ elements. A polynomial $f \in \Bbb F_q[{\tt x}]$ is called a \textit{permutation polynomial} (PP) of $\Bbb F_q$ if the associated mapping $x\mapsto f(x)$ from $\f_q$ to $\f_q$ is a permutation of $\Bbb F_q$. Permutation polynomials over finite fields have important applications in Coding Theory, Cryptography, Finite Geometry, Combinatorics and Computer Science, among other fields.

Recently, reversed Dickson polynomials over finite fields have been studied extensively by many for their general properties and permutation behaviour. The concept of the reversed Dickson polynomial $D_{n}(a,x)$ was first introduced by Hou, Mullen, Sellers and Yucas in \cite { Hou-Mullen-Sellers-Yucas-FFA-2009} by reversing the roles of the variable and the parameter in the Dickson polynomial $D_{n}(x,a)$. 

The $n$-th reversed Dickson polynomial of the first kind $D_n(a,x)$ is defined by

\begin{equation}\label{E1.2}
D_{n}(a,x) = \sum_{i=0}^{\lfloor\frac n2\rfloor}\frac{n}{n-i}\dbinom{n-i}{i}(-x)^{i}a^{n-2i},
\end{equation}

where $a\in \f_q$ is a parameter. 

In \cite { Hou-Mullen-Sellers-Yucas-FFA-2009}, it was shown that the  reversed Dickson polynomials of the first kind are closely related to Almost Perfect Nonlinear (APN) functions which have applications in cryptography. Hou and Ly found more properties of  the reversed Dickson polynomials of the first kind and  necessary conditions for them to be a permutation of $\f_q$; see \cite{Hou-Ly-FFA-2010}.

By reversing the roles of the variable and the parameter in the Dickson polynomial of the second kind $E_{n}(x,a)$, the $n$-th reversed Dickson polynomial of the second kind $E_n(a,x)$ can be defined by

\begin{equation}\label{E1.4}
E_{n}(a,x) = \sum_{i=0}^{\lfloor\frac n2\rfloor}\dbinom{n-i}{i}(-x)^{i}a^{n-2i},
\end{equation}

where $a\in \f_q$ is a parameter. 

In \cite{Hong-Qin-Zhao-FFA-2016-2}, Hong, Qin, and Zhao explored the reversed Dickson polynomials of the second kind and found many of their properties and necessary conditions for them to be a permutation of $\f_q$.

For $a\in \f_q$, the $n$-th reversed Dickson polynomial of the $(k+1)$-th kind $D_{n,k}(a,x)$ is defined by
\begin{equation}\label{E1.1}
D_{n,k}(a,x) = \sum_{i=0}^{\lfloor\frac n2\rfloor}\frac{n-ki}{n-i}\dbinom{n-i}{i}(-x)^{i}a^{n-2i},
\end{equation}

and $D_{0,k}(a,x)=2-k$ ; See \cite{Wang-Yucas-FFA-2012}.

Note that  $D_{n,0}(a,x)=D_{n}(a,x)$ and $D_{n,1}(a,x)=E_{n}(a,x)$. Also note that we only need to consider $0\leq k\leq p-1$. It follows from \eqref{E1.2} , \eqref{E1.4}, and \eqref{E1.1} that 

\begin{equation}\label{E1.3}
D_{n,k}(a,x)=kE_{n}(a,x)-(k-1)D_{n}(a,x).
\end{equation}

The author of the current paper surveyed the properties and the permutation behaviour of the reversed Dickson polynomials of the third kind over finite fields in \cite{Fernando-2016}. Most recently, Cheng, Hong, and Qin studied the reversed Dickson polynomials of the fourth kind over finite fields; see \cite{Hong-Qin-Zhao-FFA-2016-4}. 

Motivated by \cite{Hong-Qin-Zhao-FFA-2016-4}, \cite{Hong-Qin-Zhao-FFA-2016-2}, \cite{Fernando-2016}, \cite{Hou-Ly-FFA-2010}, and \cite{Hou-Mullen-Sellers-Yucas-FFA-2009}, we fix $k$ and study the general properties and permutation property of the reversed Dickson polynomials of the $(k+1)$-th kind over finite fields. The results obtained in this paper unify and generalize many existing results on reversed Dickson polynomials. 

The paper is organized as follows. In Section 2, we present the generating function and some other general properties of the reversed Dickson polynomials of the $(k+1)$-th kind. We also discuss the cases $a=0$, $n=p^l$, $n=p^l+1$, and $n=p^l+2$, where $l\geq 0$ is an integer. Some necessary and sufficient conditions for $D_{n,k}(1,x)$ to be a permutation of $\f_q$ are presented in Section 2 as well. 

In Section 3, we give an explicit expression for the $n$-th reversed Dickson polynomial of the $(k+1)$-th kind $D_{n,k}(1,x)$. 

In Section 4, we compute the sum $\sum_{a\in \f_q}D_{n,k}(1,a)$. 


\section{Reversed Dickson polynomials of the $(k+1)$-th kind}

\subsection{The Case $a=0$}

When $a=0$, the reversed Dickson polynomials of the first kind satisfy (See \cite{Hou-Mullen-Sellers-Yucas-FFA-2009})

\[
D_n (0,x)=
\begin{cases}
0&\text{if}\ n\,\, \text{is odd},\cr
2\,(-x)^l&\text{if}\ n = 2l,
\end{cases}
\]

and the reversed Dickson polynomials of the second kind satisfy (See \cite{Hong-Qin-Zhao-FFA-2016-2})

\[
E_n (0,x)=
\begin{cases}
0&\text{if}\ n\,\, \text{is odd},\cr
(-x)^l&\text{if}\ n = 2l.
\end{cases}
\]

\eqref{E1.3} implies 

\[
D_{n,k} (0,x)=
\begin{cases}
0&\text{if}\ n\,\, \text{is odd},\cr
(2-k)\,(-x)^l&\text{if}\ n = 2l.
\end{cases}
\]

\begin{thm}\label{T2.2}
When $a=0$, $D_{n,k}(a,x)$ is a PP of $\f_q$ if and only if $k\neq 2$ and $n=2l$ with $(l,q-1)=1$. 
\end{thm}

\begin{rmk}
Let $k=0$. Then Theorem~\ref{T2.2} explains the results in \cite[Section 2]{Hou-Mullen-Sellers-Yucas-FFA-2009}.
\end{rmk}

\begin{rmk}
Let $k=1$. Then Theorem~\ref{T2.2} explains the results in \cite[Section 2]{Hong-Qin-Zhao-FFA-2016-2}. 
\end{rmk}

\begin{rmk}
Let $k=2$. Then Theorem~\ref{T2.2} implies $D_{n,2}(0,x)$ is not a PP of $\f_q$ which is \cite[Lemma 2.1]{Fernando-2016}.
\end{rmk}

\begin{rmk}
Let $k=3$, $p>3$, and $l=\frac{n}{2}$. Then Theorem~\ref{T2.2} explains the results in \cite[Section 2]{Hong-Qin-Zhao-FFA-2016-4}. 
\end{rmk}

We thus hereafter assume that $a\in \f_{q}^{*}$. 

For $a\neq 0$, we write $x=y(a-y)$ with an indeterminate $y\in \f_{q^2}$ such that $y\neq \frac{a}{2}$.  Then 

\[
D_{n}(a,x)=y^n+(a-y)^n
\]

and 

\[
E_n(a,x)=\displaystyle\frac{y^{n+1}-(a-y)^{n+1}}{2y-a}
\]

are the functional expressions of the Dickson polynomial of the first kind and second kind, respectively. From \eqref{E1.3}, we have 

\[
\begin{split}
D_{n,k}(a,x)=k\,\Big[ \displaystyle\frac{y^{n+1}-(a-y)^{n+1}}{2y-a}\Big]-(k-1)\{y^n+(a-y)^n\},
\end{split}
\]

where $y\neq \frac{a}{2}$. Note that

\[
\begin{split}
D_{n,k}(a,x)&=k\,\Big[ \displaystyle\frac{y^{n+1}-(a-y)^{n+1}}{2y-a}\Big]-(k-1)\{y^n+(a-y)^n\}\cr
&=k\,\Big[ \displaystyle\frac{y^{n+1}-(a-y)^{n+1}}{2y-a}-y^n-(a-y)^n\Big]+y^n+(a-y)^n\cr
&=k\,\Big[ \displaystyle\frac{y^n(a-y)-y(a-y)^n}{2y-a}\Big]+y^n+(a-y)^n,
\end{split}
\]

i.e.

\begin{equation}\label{E2.1}
\begin{split}
D_{n,k}(a,x)&=k\,\Big[ \displaystyle\frac{y^n(a-y)-y(a-y)^n}{2y-a}\Big]+D_n(a,x).
\end{split}
\end{equation}

It follows from the definitions that in characteristic 2, 

\[
D_{n,k} (a,x)=
\begin{cases}
E_n(a,x) &\text{if}\ k\,\, \text{is odd},\cr
D_n(a,x) &\text{if}\ k\,\, \text{is even}.
\end{cases}
\]

We thus hereafter always assume, unless specified, in this paper that $p$ is odd.  

Let  $a\in \f_{q}^{*}$. Then it follows from the definition that
\begin{equation}\label{E2.2}
D_{n,k}(a,x)=a^n\,D_{n,k}(1,\frac{x}{a^2}). 
\end{equation}
Hence $D_{n,k}(a,x)$ is a PP on $\f_{q}$ if and only if $D_{n,k}(1,x)$ is a PP on $\f_{q}$.

From \eqref{E2.1}, we have 

\begin{equation}\label{E2.3}
\begin{split}
D_{n,k}(1,y(1-y))&=k\,\Big[ \displaystyle\frac{y^n(1-y)-y(1-y)^n}{2y-1}\Big]+D_n(1,y(1-y)),
\end{split}
\end{equation}

where $y\neq \frac{1}{2}$. 

Let $x=1$ and $a=2$ in \eqref{E2.2} to obtain 

\begin{equation}\label{E2.4}
D_{n,k}\Big(1,\frac{1}{4}\Big)\,=\,\frac{D_{n,k}(2,1)}{2^n}. 
\end{equation}

Since $E_n(2,1)=n+1$ and $D_n(2,1)=2$ (See \cite{Lidl-Mullen-Turnwald-1993}), we have from \eqref{E1.3}

\begin{equation}\label{E2.5}
D_{n,k}(2,1)=kE_{n}(2,1)-(k-1)D_{n}(2,1)=k(n-1)+2.
\end{equation}

Combining \eqref{E2.4} and \eqref{E2.5} we get for $y=\frac{1}{2}$

\begin{equation}\label{E2.6}
D_{n,k}\Big(1,\frac{1}{4}\Big)\,=\,\frac{k(n-1)+2}{2^n}. 
\end{equation}

It is clear from \eqref{E2.3} that if $n_1\equiv n_2 \pmod{q^2-1}$, then $D_{n_1,k}(1,x)=D_{n_2,k}(1,x)$ for any $x\in \f_{q}\setminus \{\frac{1}{4}\}$.

\begin{prop}\label{P2.2} Let $p$ be an odd prime and $n$ be a non-negative integer. Then 
$$D_{0,k}(1,x)=2-k,\,\,D_{1,k}(1,x)=1,\,\,\textnormal{and}$$
$$D_{n,k}(1,x)=D_{n-1,k}(1,x)-x\,D_{n-2,k}(1,x),\,\, \textnormal{for} \,\,n\geq 2.$$
\end{prop}
\begin{proof}
It follows from \eqref{E2.3} and \eqref{E2.6} that $D_{0,k}(1,x)=2-k$ and $D_{1,k}(1,x)=1$. 

Let $n\geq 2$.

When $x=\frac{1}{4}$, 
\[
\begin{split}
D_{n-1,k}\Big(1,\frac{1}{4}\Big)-\frac{1}{4}\,D_{n-2,k}\Big(1,\frac{1}{4}\Big)&=\frac{k(n-2)+2}{2^{n-1}}-\frac{1}{4}\frac{k(n-3)+2}{2^{n-2}}=\frac{k(n-1)+2}{2^{n}}\cr
&=D_{n,k}\Big(1,\frac{1}{4}\Big).
\end{split}
\]

When $x\neq \frac{1}{4}$, we write $x=y(1-y)$ with $y\neq \frac{1}{2}$. The rest of the proof follows from \eqref{E2.3} and noticing that 

\[
\begin{split}
\displaystyle\frac{y^{n-1}(1-y)-y(1-y)^{n-1}}{2y-1}-y(1-y)\displaystyle\frac{y^{n-2}(1-y)-y(1-y)^{n-2}}{2y-1}\cr
=\displaystyle\frac{y^n(1-y)-y(1-y)^n}{2y-1}
\end{split}
\]

and

\[
\begin{split}
D_{n-1}(1,y(1-y))-y(1-y)D_{n-2}(1,y(1-y))=D_n(1,y(1-y)).
\end{split}
\]

\end{proof}

\subsection{The Case $n=p^l$, $l\geq 0$ is an integer}

Let $x=y(1-y)$. From \eqref{E2.3}, for $y\neq \frac{1}{2}$ we have
\[
\begin{split}
D_{p^l,k}(1,y(1-y))&=k\,\Big[ \displaystyle\frac{y^{p^l}(1-y)-y(1-y)^{p^l}}{2y-1}\Big]+D_{p^l}(1,y(1-y))\cr
&=k\,\displaystyle\frac{y^{p^l}-y}{2y-1}\,+\,(D_1(1,y(1-y))^{p^l}\cr
&=k\,\displaystyle\frac{y^{p^l}-y}{2y-1}\,+\,1
\end{split}
\]

Then 

\[
\begin{split}
D_{p^l,k}(1,y(1-y))&=k\,\displaystyle\frac{y^{p^l}-y}{2y-1}\,+\,1\cr
&=\frac{k}{2}\,(2y-1)^{p^l-1}+1-\frac{k}{2}\cr
&=\frac{k}{2}\,\Big((2y-1)^2\Big)^{\frac{p^l-1}{2}}+1-\frac{k}{2}\cr
\end{split}
\]

Note that $(2y-1)^2=1-4x$. So

\[
\begin{split}
D_{p^l,k}(1,x)&=\frac{k}{2}\,(1-4x)^{\frac{p^l-1}{2}}+1-\frac{k}{2}\cr
\end{split}
\]

When $y= \frac{1}{2}$, 

$$D_{p^l,k}\Big(1,\frac{1}{4}\Big)\,=\,\frac{k(p^l-1)+2}{2^{p^l}}=\frac{2-k}{2}=1-\frac{k}{2}=\frac{k}{2}\,(1-4x)^{\frac{p^l-1}{2}}+1-\frac{k}{2}.$$ 

Hence for all $x\in \f_q$, we have 

\begin{equation}\label{E2.7}
\begin{split}
D_{p^l,k}(1,x)&=\frac{k}{2}\,(1-4x)^{\frac{p^l-1}{2}}+1-\frac{k}{2}\cr
\end{split}
\end{equation}

\begin{lem}\label{L2.1}(see \cite{Lidl-Niederreiter-97})
The monomial $x^n$ is a PP of $\f_q$ if and only if $(n, q-1)=1.$
\end{lem}

\begin{thm}\label{T2.1}
Let  $0<l\leq e$. Then $D_{3^l,k}(1,x)$ is a PP of $\f_{3^e}$ if and only if $k\neq 0$ and $( \frac{3^l-1}{2}, 3^e-1)=1.$ Also, $D_{p^l,k}(1,x)$ is not a PP of $\f_{p^e}$ when $p>3$. 
\end{thm}

\begin{proof}
Clearly, if $k=0$, then $D_{3^l,k}(1,x)$ is not a PP of $\f_{3^e}$. Rest of the proof follows from lemma~\ref{L2.1} and the fact that $1-4x$ is a PP of $\f_{3^e}$. 

The second part of the proof follows from the fact that $( \frac{p^l-1}{2}, p^e-1)\neq 1$ when $p>3$ for $0<l\leq e$.
\end{proof}

\begin{rmk}
Let $k=2$ and $p$ be an odd prime. Then from \eqref{E2.7}, $D_{p^l,2}(1,x)=(1-4x)^{\frac{p^l-1}{2}}$ is a PP of $\f_q$ if and only if and $\Big( \frac{p^l-1}{2}, q-1\Big)=1$; see \cite[Theorem 2.6]{Fernando-2016}.
\end{rmk}

\begin{rmk}
Let $k=3$ and $p>3$. Then  from \eqref{E2.7}, $2\,D_{p^l,k}(1,x)=3 \,(1-4x)^{\frac{p^l-1}{2}}-1$, see \cite[Proposition 2.7]{Hong-Qin-Zhao-FFA-2016-4}. Also,  \cite[Corollary 2.8]{Hong-Qin-Zhao-FFA-2016-4} follows from the second part of the Theorem~\ref{T2.1} when $k=3$. 
\end{rmk}

\subsection{The Case $n=p^l+1$, $l\geq 0$ is an integer}

Let $x=y(1-y)$. From \eqref{E2.3}, for $y\neq \frac{1}{2}$ we have
\begin{equation}\label{E2.8}
\begin{split}
D_{p^l+1,k}(1,x)&= D_{p^l+1,k}(1,y(1-y))\cr
&=k\,\Big[ \displaystyle\frac{y^{p^l+1}(1-y)-y(1-y)^{p^l+1}}{2y-1}\Big]+D_{p^l+1}(1,y(1-y))\cr
&=k\,\{y(1-y)(2y-1)^{p^l-1}\} + D_{p^l+1}(1,y(1-y))\cr
\end{split}
\end{equation}

Let $u=2y-1$. Then it follows from \eqref{E2.8}

\[
\begin{split}
D_{p^l+1,k}(1,x)&= D_{p^l+1,k}(1,y(1-y))\cr
&=k\,\{y(1-y)(2y-1)^{p^l-1}\} + D_{p^l+1}(1,y(1-y))\cr
&=\frac{k}{4}\,(1-u^2)\,u^{p^l-1} + \frac{1}{2}\,(u^{p^l+1}+1)\cr
&=\Big(\frac{1}{2} - \frac{k}{4}\Big)\,u^{p^l+1}+\frac{k}{4}\,u^{p^l-1} +\frac{1}{2}\cr
&=\Big(\frac{1}{2} - \frac{k}{4}\Big)\,(u^2)^{\frac{p^l+1}{2}}+\frac{k}{4}\,(u^2)^{\frac{p^l-1}{2}} +\frac{1}{2}\cr
&=\Big(\frac{1}{2} - \frac{k}{4}\Big)\,(1-4x)^{\frac{p^l+1}{2}}+\frac{k}{4}\,(1-4x)^{\frac{p^l-1}{2}} +\frac{1}{2}
\end{split}
\]

When $y= \frac{1}{2}$,

$$D_{p^l+1,k}\Big(1,\frac{1}{4}\Big)\,=\,\frac{kp^l+2}{2^{p^l+1}}=\frac{1}{2}=\Big(\frac{1}{2} - \frac{k}{4}\Big)\,(1-4x)^{\frac{p^l+1}{2}}+\frac{k}{4}\,(1-4x)^{\frac{p^l-1}{2}} +\frac{1}{2}.$$ 

Hence for all $x\in \f_q$, we have 

\begin{equation}\label{E2.9}
\begin{split}
D_{p^l+1,k}(1,x)&=\Big(\frac{1}{2} - \frac{k}{4}\Big)\,(1-4x)^{\frac{p^l+1}{2}}+\frac{k}{4}\,(1-4x)^{\frac{p^l-1}{2}} +\frac{1}{2}
\end{split}
\end{equation}

\begin{rmk}
Let $k=0$. Then, from \eqref{E2.9} $D_{p^l+1,k}(1,x)=\frac{1}{2}\,(1-4x)^{\frac{p^l+1}{2}}+\frac{1}{2}$ which is a PP of $\f_q$ if and only if $\Big( \frac{p^l+1}{2}, q-1\Big)=1$ (see \cite[Proposition 5.1, Corollary 5.2]{Hou-Mullen-Sellers-Yucas-FFA-2009}). 
\end{rmk}

\begin{thm}
Let $k=2$. Then $D_{p^l+1,k}(1,x)$ is a PP of $\f_q$ if and only if $\Big( \frac{p^l-1}{2}, q-1\Big)=1$. 
\end{thm}

\begin{proof}
When $k=2$, from \eqref{E2.9} we have $D_{p^l+1,k}(1,x)=\frac{1}{2}\,(1-4x)^{\frac{p^l-1}{2}}+\frac{1}{2}$  which is a PP of $\f_q$ if and only if $\Big( \frac{p^l-1}{2}, q-1\Big)=1$. 
\end{proof}

\begin{thm}
Let $n=p^l+1$ and $k\neq 0, 2$. Then $D_{n,k}(1,x)$ is a PP of $\f_q$ if and only if $l=0$. 
\end{thm}

\begin{proof}
When $k\neq 0, 2$, from \eqref{E2.9} we have $D_{n,k}(1,x)$ is a PP of $\f_q$ if and only if the binomial $(2-k)\,x^{\frac{p^l+1}{2}}+k\,x^{\frac{p^l-1}{2}}$ is a PP of $\f_q$. Hence the proof follows from the fact that $(2-k)\,x^{\frac{p^l+1}{2}}+k\,x^{\frac{p^l-1}{2}}$ is a PP of $\f_q$ if and only if $l=0$. 
\end{proof}

\subsection{The Case $n=p^l+2$, $l\geq 0$ is an integer}

From \eqref{E2.7}, \eqref{E2.9}, and Proposition~\ref{P2.2} we have

\begin{equation}\label{E2.10}
\begin{split}
D_{p^l+2,k}(1,x)&=D_{p^l+1,k}(1,x)-xD_{p^l,k}(a,x)\cr
&= \frac{1}{2}\,(1-4x)^{\frac{p^l+1}{2}}+\frac{k}{2}\,x\,(1-4x)^{\frac{p^l-1}{2}}-\Big(1-\frac{k}{2}\big)x+\frac{1}{2}
\end{split}
\end{equation}

\begin{rmk}
Let $k=0$ and $l=e$. Then from \eqref{E2.10} we have
\[
\begin{split}
D_{p^e+2,0}(1,x)&= \frac{1}{2}\,(1-4x)^{\frac{p^e+1}{2}}-x+\frac{1}{2}
\end{split}
\]
which is a PP of $\f_{p^e}$ if and only if $p^e\equiv 1 \pmod{3}$ (see \cite[Corollary 5.2]{Hou-Mullen-Sellers-Yucas-FFA-2009}).
\end{rmk}

let $u=1-4x$ in \eqref{E2.10}. Then we have

\begin{equation}\label{E2.11}
\begin{split}
D_{p^l+2,k}(1,x)&= \frac{1}{2}\,(1-4x)^{\frac{p^l+1}{2}}+\frac{k}{2}\,x\,(1-4x)^{\frac{p^l-1}{2}}-\Big(1-\frac{k}{2}\big)x+\frac{1}{2}\cr
&=\frac{1}{2}\,u^{\frac{p^l+1}{2}}+\frac{k}{2}\,\Big(\frac{1-u}{4}\Big)\,u^{\frac{p^l-1}{2}}-\Big(\frac{1-u}{4}\Big)\Big(1-\frac{k}{2}\Big)+\frac{1}{2}\cr
&=\Big(\frac{1}{2}-\frac{k}{8}\Big)\,u^{\frac{p^l+1}{2}}+\frac{k}{8}\,u^{\frac{p^l-1}{2}}+\Big(1-\frac{k}{2}\Big)\frac{u}{4}+\frac{k}{8}+\frac{1}{4}
\end{split}
\end{equation}

\begin{thm}
Let $k=2$. Then $D_{p^l+2,k}(1,x)$ is a PP of $\f_q$ if and only if $l=0$. 
\end{thm}

\begin{proof}
When $k=2$, It follows from \eqref{E2.11} that $D_{p^l+2,k}(1,x)$ is a PP of $\f_q$ if and only if the binomial $x^{\frac{p^l+1}{2}}+x^{\frac{p^l-1}{2}}$ is a PP of $\f_q$. Note that $x^{\frac{p^l+1}{2}}+x^{\frac{p^l-1}{2}}$ is a PP of $\f_q$ if and only if $l=0$. 
\end{proof}

\begin{thm}
Let $p>3$ and $k=4$. Then it follows from \eqref{E2.11} that $D_{p^l+2,k}(1,x)$ is a PP of $\f_q$ if and only if the binomial $x^{\frac{p^l-1}{2}}-\frac{1}{2}x$ is a PP of $\f_q$. 
\end{thm}

\begin{proof}
When $p>3$ and $k=4$, It follows from \eqref{E2.11} that $D_{p^l+2,k}(1,x)$ is a PP of $\f_q$ if and only if the binomial $x^{\frac{p^l-1}{2}}-\frac{1}{2}x$  is a PP of $\f_q$. Note that $x^{\frac{p^l-1}{2}}-\frac{1}{2}x$  is a PP of $\f_q$ if and only if $l=0$. 
\end{proof}

\begin{thm}
Let $n=p^l+2$ and $k\neq 0, 2, 4$. Then $D_{n,k}(1,x)$ is a PP of $\f_q$ if and only if the trinomial $(4-k)\,x^{\frac{p^l+1}{2}}+k\,x^{\frac{p^l-1}{2}}+(2-k)x$ is a PP of $\f_q$. 
\end{thm}

\subsection{The generating function}The generating function of $D_{n,k}(1,x)$ is given by 
$$\displaystyle\sum_{n=0}^{\infty}\,D_{n,k}(1,x)\,z^n=\displaystyle\frac{2-k+(k-1)z}{1-z+xz^2}.$$

\begin{proof}
\[
\begin{split}
&(1-z+xz^2)\displaystyle\sum_{n=0}^{\infty}\,D_{n,k}(1,x)\,z^n \cr
&=\displaystyle\sum_{n=0}^{\infty}\,D_{n,k}(1,x)\,z^n-\,\displaystyle\sum_{n=0}^{\infty}\,D_{n,k}(1,x)\,z^{n+1}+\,x\,\displaystyle\sum_{n=0}^{\infty}\,D_{n,k}(1,x)\,z^{n+2}\cr
&=D_{0,k}(1,x)+D_{1,k}(1,x)z-D_{0,k}(1,x)z \cr &+\displaystyle\sum_{n=0}^{\infty}\,(D_{n+2,k}(1,x)-D_{n+1,k}(1,x)+xD_{n,k}(1,x))\,z^{n+2}
\end{split}
\]
Since $D_{0,k}(1,x)=2-k$, $D_{1,k}(1,x)=1$, and $D_{n+2,k}(1,x)=D_{n+1,k}(1,x)-xD_{n,k}(1,x)$ for $n\geq 0$, we have the desired result. 
\end{proof}

\begin{rmk}
When $k=1,\, \textnormal{and}\,\, 2$, we get \cite[Proposition 2.2]{Hong-Qin-Zhao-FFA-2016-2} and \cite[Theorem 2.10]{Fernando-2016}, respectively. When $k=3$ with $p>3$, we get \cite[Proposition 2.6]{Hong-Qin-Zhao-FFA-2016-4}. The generating function when $k=0$ appears in \cite[Section 4]{Hou-Ly-FFA-2010}. 
\end{rmk}

\begin{lem}\label{L2.17}(See \cite{Hou-Mullen-Sellers-Yucas-FFA-2009})
Let $q=p^e$ and Let $x\in \f_{q^2}$. Then 
$$ x(1-x)\in \f_q \,\,\textnormal{if and only if} \,\,x^q=x \,\,\textnormal{or}\,\, x^q=1-x.$$
Also, if we define
$$V=\{x\in \f_{q^2}\,;\,x^q=1-x\},$$
then $\f_q \cap V=\frac{1}{2}$. 
\end{lem}

\begin{thm}
Let $p$ be an odd prime. Then $D_{n,k}(1,x)$ is a PP of $\f_q$ if and only if the function $y \mapsto k\,\displaystyle\frac{y^n(1-y)-y(1-y)^n}{2y-1}+y^n+(1-y)^n$ is a 2-to-1 mapping on $(\f_q \cup V)\setminus \frac{1}{2}$ and $k\,\displaystyle\frac{y^n(1-y)-y(1-y)^n}{2y-1}+y^n+(1-y)^n\neq \frac{k(n-1)+2}{2^n}$ for any $y\in (\f_q \cup V)\setminus \frac{1}{2}$. 
\end{thm}

\begin{proof}
For necessity, assume that $D_{n,k}(1,x)$ is a PP of $\f_q$ and $y_1, y_2\in (\f_q \cup V)\setminus \frac{1}{2}$ such that 
$$k\,\displaystyle\frac{y_1^n(1-y_1)-y(1-y_1)^n}{2y_1-1}+y_1^n+(1-y_1)^n=k\,\displaystyle\frac{y_2^n(1-y_2)-y(1-y_2)^n}{2y_2-1}+y_2^n+(1-y_2)^n.$$ Then $y_1(1-y_1), y_2(1-y_2) \in \f_q$ and $D_{n,k}(1,y_1(1-y_1))=D_{n,k}(1,y_2(1-y_2))$. Since $D_{n,k}(1,x)$ is a PP of $\f_q$, we have $y_1(1-y_1)=y_2(1-y_2)$ which implies that $y_1=y_2$ or $1-y_2$. So $y \mapsto \displaystyle\frac{y^n-(1-y)^n}{2y-1}$ is a 2-to-1 mapping on $(\f_q \cup V)\setminus \frac{1}{2}$. 

If $y\in (\f_q \cup V)\setminus \frac{1}{2}$, then $y(1-y)\in \f_q$ and $y(1-y)\neq \frac{1}{2}(1-\frac{1}{2})$. Thus 
\[
\begin{split}
k\,\displaystyle\frac{y^n(1-y)-y(1-y)^n}{2y-1}+y^n+(1-y)^n &=D_{n,k}(1,y(1-y))\cr
&\neq D_{n,k}(1,\frac{1}{2}(1-\frac{1}{2}))=\frac{k(n-1)+2}{2^n}. 
\end{split}
\]

For sufficiency, assume $x_1, x_2 \in \f_q$ such that $D_{n,k}(1,x_1)=D_{n,k}(1,x_2)$. Write $x_1=y_1(1-y_1)$ and $x_2=y_2(1-y_2)$, where $y_1, y_2 \in (\f_q \cup V)$. Then 
\[
\begin{split}
k\,\displaystyle\frac{y_1^n(1-y_1)-y(1-y_1)^n}{2y_1-1}+y_1^n+(1-y_1)^n =D_{n,k}(1,x_1)=D_{n,k}(1,x_2)\cr
=k\,\displaystyle\frac{y_2^n(1-y_2)-y(1-y_2)^n}{2y_2-1}+y_2^n+(1-y_2)^n.
\end{split}
\]

If $y_1=\frac{1}{2}$, then 
$$D_{n,k}(1,x_2)=D_{n,k}(1,x_1)=D_{n,k}(1,\frac{1}{4})=\frac{k(n-1)+2}{2^n},$$
which implies that $y_2=\frac{1}{2}$. Hence $x_1=x_2$. 

If $y_1, y_2 \neq \frac{1}{2}$, since  $y \mapsto k\,\displaystyle\frac{y^n(1-y)-y(1-y)^n}{2y-1}+y^n+(1-y)^n$ is a 2-to-1 mapping on $(\f_q \cup V)\setminus \frac{1}{2}$, we have $y_1=y_2$ or $y_1=1-y_2$. Hence $x_1=x_2$. 

\end{proof}

\begin{rmk}
When $k=0,1,2,$ and $3\,(p>3)$, we get \cite[Proposition 4.2]{Hou-Mullen-Sellers-Yucas-FFA-2009}, \cite[Theorem 2.3]{Hong-Qin-Zhao-FFA-2016-2},  \cite[Theorem 2.10]{Fernando-2016}, and \cite[Theorem 2.10]{Hong-Qin-Zhao-FFA-2016-4}, respectively. 
\end{rmk}


\section{An explicit expression for $D_{n,k}(1,x)$}

Let $p$ be odd. The $n$-th reversed Dickson polynomial of the $(k+1)$-th kind $D_{n,k}(1,x)$ can be written explicitly. For $n\geq 1$, define
$$f_{n,k}(x)=k\,\,\displaystyle\sum_{j\geq 0} \,\,\binom{n-1}{2j+1}\,\,(x^j-x^{j+1})+2\,\,\displaystyle\sum_{j\geq 0}\,\,\binom{n}{2j}\,\,x^j \,\,\in \mathbb{Z}[x], $$

and

$$f_{0,k}(x)=2-k.$$

We show that for $n\geq 0$

\begin{equation}\label{N1}
D_{n,k}(1,x)=\Big(\frac{1}{2}\Big)^{n}\,f_{n,k}(1-4x).
\end{equation}

Note that \eqref{N1} holds for $n=0$ for any $k$. We thus hereafter always assume, unless specified, that $n>0$.  

\begin{prop}\label{P3.2}
Let $p$ be an odd prime and $n>0$ be an integer. Then in $\f_q[x]$, 
$$D_{n,k}(1,x)=\Big(\frac{1}{2}\Big)^{n}\,f_{n,k}(1-4x).$$
In particular, $D_{n,k}(1,x)$ is a PP of $\f_q$ if and only if $f_{n,k}(x)$ is a PP of $\f_q$. 
\end{prop}

\begin{proof}
Let $x\in \f_q$. There exists $y\in \f_{q^2}$ such that $x=y(1-y)$. 
If $x\neq \frac{1}{4}$, we have 
\begin{equation}\label{E3.3}
\begin{split}
D_{n,k}(1,y(1-y))=k\,\Big[ \displaystyle\frac{y^n(1-y)-y(1-y)^n}{2y-1}\Big]+D_n(1,y(1-y)).
\end{split}
\end{equation}

Let $u=2y-1$. Then we have 
$$D_n(1,y(1-y))=\Big(\frac{1}{2}\Big)^{n-1}\,\sum_{j\geq 0}\,\binom{n}{2j}\,u^{2j};$$

see \cite[Proposition 2.1]{Hou-Ly-FFA-2010}. 

We also have 

\[
\begin{split}
\Big[ \displaystyle\frac{y^n(1-y)-y(1-y)^n}{2y-1}\Big]&=\frac{1}{u}\Big\{\Big(\frac{1+u}{2}\Big)^n\Big(\frac{1-u}{2}\Big)-\Big(\frac{1+u}{2}\Big)\Big(\frac{1-u}{2}\Big)^n\Big\}\cr
&=\Big(\frac{1}{2}\Big)^{n+1}\,\, \frac{(1-u^2)}{u}\,\, \Big\{(1+u)^{n-1}-(1-u)^{n-1}\Big\}\cr
&=\Big(\frac{1}{2}\Big)^{n}\,\,(1-u^2)\,\,\displaystyle\sum_{j\geq 0}\,\binom{n-1}{2j+1}\,\,u^{2j}\cr
&=\Big(\frac{1}{2}\Big)^{n}\,\,\displaystyle\sum_{j\geq 0}\,\binom{n-1}{2j+1}\,\,u^{2j}-\Big(\frac{1}{2}\Big)^{n}\,\,\displaystyle\sum_{j\geq 0}\,\binom{n-1}{2j+1}\,\,u^{2j+2}.
\end{split}
\]

From \eqref{E3.3}, we have 

\begin{equation}\label{E3.4}
\begin{split}
D_{n,k}(1,y(1-y))&=\frac{k}{2^n}\,\,\Big\{\displaystyle\sum_{j\geq 0}\,\binom{n-1}{2j+1}\,\,u^{2j}-\displaystyle\sum_{j\geq 0}\,\binom{n-1}{2j+1}u^{2j+2}\Big\}\cr
&+ \Big(\frac{1}{2}\Big)^{n-1}\,\sum_{j\geq 0}\,\binom{n}{2j}\,u^{2j}.
\end{split}
\end{equation}

Note that $u^2=1-4y(y-1)=1-4x$.

\[
\begin{split}
D_{n,k}(1,x)&=\Big(\frac{1}{2}\Big)^{n}\,\,f_{n,k}(u^2)\cr
&=\Big(\frac{1}{2}\Big)^{n}\,\,f_{n,k}(1-4y(y-1))\cr
&=\Big(\frac{1}{2}\Big)^{n}\,\,f_{n,k}(1-4x).
\end{split}
\]

If $x=\frac{1}{4}$, since $f_{n,k}(0)=k(n-1)+2$, we have
$$D_{n,k}(1,x)=\frac{k(n-1)+2}{2^n}=\Big(\frac{1}{2}\Big)^{n}\,f_{n,k}(0)=\Big(\frac{1}{2}\Big)^{n}\,f_{n,k}(1-4x).$$

Clearly, $D_{n,k}(1,x)$ is a PP of $\f_q$ if and only if $f_{n,k}(x)$ is  PP of $\f_q$. 
\end{proof}

\begin{rmk} Let $k=0$ and $n\geq 1$. 

From Proposition~\ref{P3.2}, we have

\[
\begin{split}
D_{n,0}(1,y(1-y))&=\Big(\frac{1}{2}\Big)^{n}\,f_{n,0}(1-4x)\cr
&=\Big(\frac{1}{2}\Big)^{n}\,f_{n,0}(u^2)\cr
&=\Big(\frac{1}{2}\Big)^{n-1}\,\displaystyle\sum_{j\geq 0}\,\binom{n}{2j}\,\,u^{2j}.
\end{split}
\]

Hence

\begin{equation}\label{N2}
\begin{split}
D_{n,0}(1,x)&=\Big(\frac{1}{2}\Big)^{n-1}\,f_{n}(1-4x);
\end{split}
\end{equation}

see \cite[Proposition~2.1]{Hou-Ly-FFA-2010}, where 

$$f_{n}(x)=\displaystyle\sum_{j\geq 0} \,\,\binom{n}{2j}\,\,x^j;$$

see \cite[Eq.~2.1]{Hou-Ly-FFA-2010}.

Note that \eqref{N2} also holds for $n=0$. 

\end{rmk}

\begin{rmk} Let $k=1$ and $n\geq 1$. 

From Proposition~\ref{P3.2}, we have

\begin{equation}\label{N3}
\begin{split}
D_{n,1}(1,y(1-y))&=\Big(\frac{1}{2}\Big)^{n}\,f_{n,1}(1-4x)\cr
&=\Big(\frac{1}{2}\Big)^{n}\,f_{n,1}(u^2)\cr
&=\Big(\frac{1}{2}\Big)^{n}\,\displaystyle\sum_{j\geq 0} \,\,\binom{n-1}{2j+1}\,\,(u^{2j}-u^{2j+2})+\frac{1}{2^{n-1}}\,\,\displaystyle\sum_{j\geq 0}\,\,\binom{n}{2j}\,\,u^{2j}.
\end{split}
\end{equation}

From Pascal's identity, we have

\begin{equation}\label{E3.5}
\begin{split}
\sum_{j\geq 0}\,\binom{n+1}{2j+1}\,u^{2j}&=\sum_{j\geq 0}\,\binom{n}{2j+1}\,u^{2j}+\sum_{j\geq 0}\,\binom{n}{2j}\,u^{2j}\cr
&=\sum_{j\geq 0}\,\binom{n-1}{2j+1}\,u^{2j}+\sum_{j\geq 0}\,\binom{n-1}{2j}\,u^{2j}+\sum_{j\geq 0}\,\binom{n}{2j}\,u^{2j}\cr
&=\sum_{j\geq 0}\,\binom{n-1}{2j+1}\,u^{2j}-\sum_{j\geq 0}\,\binom{n-1}{2j+1}\,u^{2j+2}+\sum_{j\geq 0}\,\binom{n-1}{2j+1}\,u^{2j+2}\cr
&+\sum_{j\geq 0}\,\binom{n-1}{2j}\,u^{2j}+\sum_{j\geq 0}\,\binom{n}{2j}\,u^{2j}\cr
&=\sum_{j\geq 0}\,\binom{n-1}{2j+1}\,u^{2j}-\sum_{j\geq 0}\,\binom{n-1}{2j+1}\,u^{2j+2}+2\sum_{j\geq 0}\,\binom{n}{2j}\,u^{2j}.\cr
\end{split}
\end{equation}

Multiplying \eqref{E3.5} by $\dfrac{1}{2^n}$ gives the following. 

\begin{equation}\label{E3.6}
\begin{split}
\dfrac{1}{2^n}\,\sum_{j\geq 0}\,\binom{n+1}{2j+1}\,u^{2j}&=\dfrac{1}{2^n}\,\sum_{j\geq 0}\,\binom{n-1}{2j+1}\,u^{2j}-\dfrac{1}{2^n}\,\sum_{j\geq 0}\,\binom{n-1}{2j+1}\,u^{2j+2}\cr
&+\dfrac{1}{2^{n-1}}\,\sum_{j\geq 0}\,\binom{n}{2j}\,u^{2j}.
\end{split}
\end{equation}

\eqref{N3} and \eqref{E3.6} imply

\[
\begin{split}
D_{n,1}(1,y(1-y))&=\dfrac{1}{2^n}\,\sum_{j\geq 0}\,\binom{n+1}{2j+1}\,u^{2j}. 
\end{split}
\]

Hence

\begin{equation}\label{N5}
\begin{split}
D_{n,1}(1,y(1-y))&=\dfrac{1}{2^n}\,f_{n+1}(1-4x),
\end{split}
\end{equation}

see \cite[Theorem~3.2]{Hong-Qin-Zhao-FFA-2016-2}, where 

$$f_{n}(x)=\displaystyle\sum_{j\geq 0} \,\,\binom{n}{2j+1}\,\,x^j;$$

see \cite{Hong-Qin-Zhao-FFA-2016-2}.

Note that \eqref{N5} also holds for $n=0$. 

\end{rmk}

\begin{rmk} Let $k=2$ and $n\geq 1$.

From Proposition~\ref{P3.2}, we have

\begin{equation}\label{N6}
\begin{split}
D_{n,2}(1,y(1-y))&=\Big(\frac{1}{2}\Big)^{n}\,f_{n,2}(1-4x)\cr
&=\Big(\frac{1}{2}\Big)^{n}\,f_{n,2}(u^2)\cr
&=\frac{1}{2^{n-1}}\,\displaystyle\sum_{j\geq 0} \,\,\binom{n-1}{2j+1}\,\,(u^{2j}-u^{2j+2})+\frac{1}{2^{n-1}}\,\,\displaystyle\sum_{j\geq 0}\,\,\binom{n}{2j}\,\,u^{2j}.
\end{split}
\end{equation}

From Pascal's identity, we have

\begin{equation}\label{E3.7}
\begin{split}
\sum_{j\geq 0}\,\binom{n}{2j+1}\,u^{2j}&=\sum_{j\geq 0}\,\binom{n-1}{2j+1}\,u^{2j}+\sum_{j\geq 0}\,\binom{n-1}{2j}\,u^{2j}\cr
&=\sum_{j\geq 0}\,\binom{n-1}{2j+1}\,u^{2j}-\sum_{j\geq 0}\,\binom{n-1}{2j+1}\,u^{2j+2}+\sum_{j\geq 0}\,\binom{n-1}{2j+1}\,u^{2j+2}\cr
&+\sum_{j\geq 0}\,\binom{n-1}{2j}\,u^{2j}\cr
&=\sum_{j\geq 0}\,\binom{n-1}{2j+1}\,u^{2j}-\sum_{j\geq 0}\,\binom{n-1}{2j+1}\,u^{2j+2}+\sum_{j\geq 0}\,\binom{n}{2j}\,u^{2j}.\cr
\end{split}
\end{equation}

Multiplying \eqref{E3.7} by $\dfrac{1}{2^{n-1}}$ gives the following. 

\begin{equation}\label{E3.8}
\begin{split}
\dfrac{1}{2^{n-1}}\,\sum_{j\geq 0}\,\binom{n}{2j+1}\,u^{2j}&=\dfrac{1}{2^{n-1}}\,\sum_{j\geq 0}\,\binom{n-1}{2j+1}\,u^{2j}-\dfrac{1}{2^{n-1}}\,\sum_{j\geq 0}\,\binom{n-1}{2j+1}\,u^{2j+2}\cr
&+\dfrac{1}{2^{n-1}}\,\sum_{j\geq 0}\,\binom{n}{2j}\,u^{2j}.
\end{split}
\end{equation}

\eqref{N6} and \eqref{E3.8} imply

\[
\begin{split}
D_{n,2}(1,y(1-y))&=\dfrac{1}{2^{n-1}}\,\sum_{j\geq 0}\,\binom{n}{2j+1}\,u^{2j}. 
\end{split}
\]

Hence

\begin{equation}\label{N8}
\begin{split}
D_{n,2}(1,y(1-y))&=\dfrac{1}{2^{n-1}}\,f_{n}(1-4x),
\end{split}
\end{equation}

see \cite[Proposition~3.2]{Fernando-2016}, where 

$$f_{n}(x)=\displaystyle\sum_{j\geq 0} \,\,\binom{n}{2j+1}\,\,x^j;$$

see\cite{Fernando-2016}.

Note that \eqref{N8} also holds for $n=0$. 

\end{rmk}

\begin{rmk} Let $k=3$, $p>3$, and $n=2l$, where $l\in \mathbb{Z}^+$. 

From Proposition~\ref{P3.2}, we have

\begin{equation}\label{N9}
\begin{split}
D_{n,3}(1,y(1-y))&=\Big(\frac{1}{2}\Big)^{n}\,f_{n,3}(1-4x)\cr
&=\Big(\frac{1}{2}\Big)^{n}\,f_{n,3}(u^2)\cr
&=\frac{3}{2^n}\,\displaystyle\sum_{j\geq 0} \,\,\binom{n-1}{2j+1}\,\,(u^{2j}-u^{2j+2})+\frac{1}{2^{n-1}}\,\,\displaystyle\sum_{j\geq 0}\,\,\binom{n}{2j}\,\,u^{2j}.
\end{split}
\end{equation}

From Pascal's identity, we have

\begin{equation}\label{E3.9}
\begin{split}
&3\,\sum_{j\geq 0}\,\binom{2l+1}{2j+1}\,u^{2j}-4\,\sum_{j\geq 0}\,\binom{2l}{2j}\,u^{2j}=3\,\sum_{j\geq 0}\,\binom{2l}{2j+1}\,u^{2j}-\,\sum_{j\geq 0}\,\binom{2l}{2j}\,u^{2j}\cr
&=3\,\sum_{j\geq 0}\,\binom{2l-1}{2j+1}\,u^{2j}+\,\,3\,\sum_{j\geq 0}\,\binom{2l-1}{2j}\,u^{2j}-\,\sum_{j\geq 0}\,\binom{2l}{2j}\,u^{2j}\cr
&=3\,\sum_{j\geq 0}\,\binom{2l-1}{2j+1}\,u^{2j}-3\,\sum_{j\geq 0}\,\binom{2l-1}{2j+1}\,u^{2j+2}+3\,\sum_{j\geq 0}\,\binom{2l-1}{2j+1}\,u^{2j+2}\cr 
&+\,\,3\,\sum_{j\geq 0}\,\binom{2l-1}{2j}\,u^{2j}-\,\sum_{j\geq 0}\,\binom{2l}{2j}\,u^{2j}\cr
&=3\,\sum_{j\geq 0}\,\binom{2l-1}{2j+1}\,u^{2j}-3\,\sum_{j\geq 0}\,\binom{2l-1}{2j+1}\,u^{2j+2}+2\,\sum_{j\geq 0}\,\binom{2l}{2j}\,u^{2j}
\end{split}
\end{equation}

Also note that 

\begin{equation}\label{E3.10}
\begin{split}
3\,\sum_{j\geq 0}\,\binom{2l+1}{2j+1}\,u^{2j}-4\,\sum_{j\geq 0}\,\binom{2l}{2j}\,u^{2j}&=3\,\sum_{j\geq 0}\,\binom{2l+1}{2j+1}\,u^{2j}-4\,\sum_{j\geq 0}\,\,\dfrac{(2j+1)}{(2l+1)}\,\binom{2l+1}{2j+1}\,u^{2j}\cr
&=\sum_{j\geq 0}\,\,\dfrac{(6l-8j-1)}{(2l+1)}\,\binom{2l+1}{2j+1}\,u^{2j}\cr
\end{split}
\end{equation}

From \eqref{E3.9} and \eqref{E3.10} we have

\begin{equation}\label{E3.11}
\begin{split}
\sum_{j\geq 0}\,\,\dfrac{(6l-8j-1)}{(2l+1)}\,\binom{2l+1}{2j+1}\,u^{2j}&=3\,\sum_{j\geq 0}\,\binom{2l-1}{2j+1}\,u^{2j}-3\,\sum_{j\geq 0}\,\binom{2l-1}{2j+1}\,u^{2j+2}\cr
&+2\,\sum_{j\geq 0}\,\binom{2l}{2j}\,u^{2j}
\end{split}
\end{equation}

Multiplying \eqref{E3.11} by $\dfrac{1}{2^{2l}}$ gives the following. 

\begin{equation}\label{E3.12}
\begin{split}
\dfrac{1}{2^{2l}}\,\,\sum_{j\geq 0}\,\,\dfrac{(6l-8j-1)}{(2l+1)}\,\binom{2l+1}{2j+1}\,u^{2j}&=\dfrac{3}{2^{2l}}\,\sum_{j\geq 0}\,\binom{2l-1}{2j+1}\,u^{2j}-\dfrac{3}{2^{2l}}\,\sum_{j\geq 0}\,\binom{2l-1}{2j+1}\,u^{2j+2}\cr
&+\dfrac{1}{2^{2l-1}}\,\sum_{j\geq 0}\,\binom{2l}{2j}\,u^{2j}
\end{split}
\end{equation}

\eqref{N9} and \eqref{E3.12} imply

\[
\begin{split}
D_{n,3}(1,y(1-y))&=\dfrac{1}{2^{2l}}\,\,\sum_{j\geq 0}\,\,\dfrac{(6l-8j-1)}{(2l+1)}\,\binom{2l+1}{2j+1}\,u^{2j}. 
\end{split}
\]

Hence

\begin{equation}\label{N10}
\begin{split}
D_{n,3}(1,y(1-y))&=\dfrac{1}{2^{n}}\,f_{n}(1-4x),
\end{split}
\end{equation}

see \cite[Theorem~3.2]{Hong-Qin-Zhao-FFA-2016-4}, where 

$$f_{n}(x)=-x^{\frac{n}{2}}+\,\,\displaystyle\sum_{j=0}^{\frac{n}{2}-1}\,\,\dfrac{(3n-8j-1)}{(n+1)}\,\binom{n+1}{2j+1}\,x^{j};$$

see \cite{Hong-Qin-Zhao-FFA-2016-4}.

Note that \eqref{N10} also holds for $n=0$. 

\end{rmk}


\section{Computation of $\sum_{a\in \f_q}D_{n,k}(1,a)$. }

We compute the sum $\sum_{a\in \f_q}D_{n,k}(1,a)$ in this section. From the generating function of $D_{n,k}(1,x)$, we have
\begin{equation}\label{E4.1}
\begin{split}
\displaystyle\sum_{n=0}^{\infty}\,D_{n,k}(1,x)\,z^n&=\displaystyle\frac{2-k+(k-1)z}{1-z+xz^2}\cr
&=\displaystyle\frac{2-k+(k-1)z}{1-z}\,\,\displaystyle\frac{1}{1-(\frac{z^2}{z-1})\,x}\cr
&=\displaystyle\frac{2-k+(k-1)z}{1-z}\,\,\displaystyle\sum_{m\geq 0} \Big(\frac{z^2}{z-1}\Big)^m\,\,x^m\cr
&=\displaystyle\frac{2-k+(k-1)z}{1-z}\,\,\Big[1+\displaystyle\sum_{m=1}^{q-1}\displaystyle\sum_{l\geq 0} \Big(\frac{z^2}{z-1}\Big)^{m+l(q-1)}\,\,x^{m+l(q-1)} \Big]\cr
&\equiv \displaystyle\frac{2-k+(k-1)z}{1-z}\,\,\Big[1+\displaystyle\sum_{m=1}^{q-1}\displaystyle\sum_{l\geq 0} \Big(\frac{z^2}{z-1}\Big)^{m+l(q-1)}\,\,x^m \Big]\,\,\,\pmod{x^q-x}\cr
&=\displaystyle\frac{2-k+(k-1)z}{1-z}\,\,\Big[1+\displaystyle\sum_{m=1}^{q-1}\displaystyle\frac{(\frac{z^2}{z-1})^{m}}{1-(\frac{z^2}{z-1})^{q-1}}\,\,x^m \Big]\cr
&=\displaystyle\frac{2-k+(k-1)z}{1-z}\,\,\Big[1+\displaystyle\sum_{m=1}^{q-1}\displaystyle\frac{(z-1)^{q-1-m}\,\,z^{2m}}{(z-1)^{q-1} - z^{2(q-1)}}\,\,x^m \Big]\cr
\end{split}
\end{equation}

Since $D_{n_1,k}(1,x)=D_{n_2,k}(1,x)$ for any $x\in \f_{q}\setminus \{\frac{1}{4}\}$ when $n_1, n_2 >0$ and $n_1\equiv n_2 \pmod{q^2-1}$, we have the following. 

\begin{equation}\label{E4.2}
\begin{split}
\displaystyle\sum_{n\geq 0} \,D_{n,k}(1,x)\,z^n&= 2-k+\displaystyle\sum_{n\geq 1} \,D_{n,k}(1,x)\,z^n\cr
&=2-k+\displaystyle\sum_{n=1}^{q^2-1}\,\, \displaystyle\sum_{l\geq 0} \,D_{n+l(q^2-1),k}(1,x)\,z^{n+l(q^2-1)}\cr
&\equiv 2-k+\displaystyle\sum_{n=1}^{q^2-1}\,\,D_{n,k}(1,x)\,\, \displaystyle\sum_{l\geq 0}\,z^{n+l(q^2-1)}\,\,\pmod{x^q-x}\cr
&=2-k+\displaystyle\frac{1}{1-z^{q^2-1}} \displaystyle\sum_{n=1}^{q^2-1}\,D_{n,k}(1,x)\,z^n
\end{split}
\end{equation}

Combining \eqref{E4.1} and \eqref{E4.2} gives

\[
\begin{split}
2-k+\displaystyle\frac{1}{1-z^{q^2-1}} \displaystyle\sum_{n=1}^{q^2-1}\,D_{n,k}(1,x)\,z^n \equiv \displaystyle\frac{2-k+(k-1)z}{1-z}\,\,\Big[1+\displaystyle\sum_{m=1}^{q-1}\displaystyle\frac{(z-1)^{q-1-m}\,\,z^{2m}}{(z-1)^{q-1} - z^{2(q-1)}}\,\,x^m \Big] \cr
\pmod{x^q-x},
\end{split}
\]

i.e.

\[
\begin{split}
\displaystyle\sum_{n=1}^{q^2-1}\,D_{n,k}(1,x)\,z^n \equiv \displaystyle\frac{z\,(z^{q^2-1}-1)}{z-1} +\,\,h(z)\,\,\displaystyle\sum_{m=1}^{q-1}(z-1)^{q-1-m}\,\,z^{2m}\,\,x^m \cr
\pmod{x^q-x},
\end{split}
\]

where

$$h(z)=\displaystyle\frac{[2-k+(k-1)z]\,(z^{q^2-1}-1)}{(z-1)\,[(z-1)^{q-1}-z^{2(q-1)}]}.$$

Note that 

\[
\begin{split}
h(z) &= \displaystyle\frac{[2-k+(k-1)z]\,(z^{q^2-1}-1)}{(z-1)\,[(z-1)^{q-1}-z^{2(q-1)}]}\cr
&= \displaystyle\frac{[2-k+(k-1)z]\,(z^{q^2-1}-1)}{(z-1)^{q}-z^{2(q-1)}(z-1)}\cr
&= \displaystyle\frac{[2-k+(k-1)z]\,(z^{q^2-1}-1)}{(1-z^{q-1})\,(z^q-z^{q-1}-1)}\cr
&=  \displaystyle\frac{[2-k+(k-1)z]\,(z^{q^2}-z)}{(z-z^{q})\,(z^q-z^{q-1}-1)}\cr
&= \displaystyle\frac{[2-k+(k-1)z]\,(-1-(z-z^q)^{q-1})}{z^q-z^{q-1}-1}\cr
\end{split}
\]

Let $\displaystyle\sum_{j=0}^{q^2-q+1}\,b_j\,z^j = [2-k+(k-1)z]\,\,(-1-(z-z^q)^{q-1})$. 

Write $j=\alpha + \beta q$ where $0\leq \alpha, \beta \leq q-1$. Then we have the following.

$$
b_j = \left\{
        \begin{array}{ll}
           (-1)^{\beta +1} \,(2-k)\,\binom{q-1}{\beta} &  \textnormal{if}\,\,\alpha +\beta =q-1,\\[0.3cm]
           (-1)^{\beta +1} \,(k-1)\,\binom{q-1}{\beta} &  \textnormal{if}\,\,\alpha +\beta =q,\\[0.3cm]
            1-k &  \textnormal{if}\,\,\alpha +\beta =1, \\[0.3cm]
            k-2 &  \textnormal{if}\,\,\alpha +\beta =0, \\[0.3cm]
            0 &  \textnormal{otherwise}.
        \end{array}
    \right.
$$

Recall that if $n_1\equiv n_2 \pmod{q^2-1}$, then $D_{n_1,k}(1,x)=D_{n_2,k}(1,x)$ for any $x\in \f_{q}\setminus \{\frac{1}{4}\}$.

\begin{equation}\label{E44.3}
\begin{split}
&\displaystyle\sum_{n=1}^{q^2-1}\, \Big(\displaystyle\sum_{a\in \f_q}\,D_{n,k}(1, a)\Big) z^n \cr 
&=\displaystyle\sum_{n=1}^{q^2-1}\,D_{n,k}(1, \frac{1}{4})\, z^n +\displaystyle\sum_{n=1}^{q^2-1}\, \Big(\displaystyle\sum_{a\in \f_q\setminus \{\frac{1}{4}\}}\,D_{n,k}(1, a)\Big) z^n\cr 
&=\displaystyle\sum_{n=1}^{q^2-1}\,D_{n,k}(1, \frac{1}{4})\, z^n +\displaystyle\sum_{a\in \f_q\setminus \{\frac{1}{4}\}}\,\displaystyle\sum_{n=1}^{q^2-1}\,D_{n,k}(1, a) z^n\cr 
&=\displaystyle\sum_{n=1}^{q^2-1}\,\frac{k(n-1)+2}{2^n}\, z^n +\displaystyle\sum_{a\in \f_q\setminus \{\frac{1}{4}\}}\,\displaystyle\frac{z\,(z^{q^2-1}-1)}{z-1}\,+\,h(z)\,\,\displaystyle\sum_{m=1}^{q-1}(z-1)^{q-1-m}\,\,z^{2m}\,\,\displaystyle\sum_{a\in \f_q\setminus \{\frac{1}{4}\}}\,a^m\cr 
&=\displaystyle\sum_{n=1}^{q^2-1}\,\frac{k(n-1)+2}{2^n}\, z^n\,-\,\displaystyle\frac{z\,(z^{q^2-1}-1)}{z-1}\,+\,h(z)\,\,\displaystyle\sum_{m=1}^{q-1}(z-1)^{q-1-m}\,\,z^{2m}\,\,\displaystyle\sum_{a\in \f_q\setminus \{\frac{1}{4}\}}\,a^m\cr
&=\displaystyle\sum_{n=1}^{q^2-1}\,\frac{k(n-1)+2}{2^n}\, z^n\,-\,\displaystyle\frac{z\,(z^{q^2-1}-1)}{z-1}\,+\,h(z)\,\,\displaystyle\sum_{m=1}^{q-1}(z-1)^{q-1-m}\,\,z^{2m}\,\,\displaystyle\sum_{a\in \f_q}\,a^m\cr
&-\,h(z)\,\,\displaystyle\sum_{m=1}^{q-1}(z-1)^{q-1-m}\,\,z^{2m}\,\,\displaystyle\sum_{a\in \f_q}\,\Big(\dfrac{1}{4}\Big)^m.
\end{split}
\end{equation}

From \eqref{E44.3} and \cite[Lemma 7.3]{Lidl-Niederreiter-97}, we have 

\begin{equation}\label{E4.3}
\begin{split}
&\displaystyle\sum_{n=1}^{q^2-1}\, \Big(\displaystyle\sum_{a\in \f_q}\,D_{n,k}(1, a)\Big) z^n \cr 
&=\displaystyle\sum_{n=1}^{q^2-1}\,\frac{k(n-1)+2}{2^n}\, z^n - \displaystyle\frac{z(1-z^{q^2-1})}{1-z} - h(z)\,z^{2(q-1)}-h(z)\displaystyle\sum_{m=1}^{q-1}\,(z-1)^{q-1-m}\,z^{2m}\,\Big(\frac{1}{4}\Big)^m,
\end{split}
\end{equation}

where 

$$h(z)= \displaystyle\frac{\displaystyle\sum_{j=0}^{q^2-q+1}\,b_jz^j}{z^q-z^{q-1}-1}.$$

For all integers $1\leq n\leq q^2-1$, define
$$\mathcal{D}_{n,k}:=\displaystyle\sum_{a\in \f_q} D_{n,k}(1,a).$$

Then from \eqref{E4.3}, we have 

\begin{equation}\label{E4.4}
\begin{split}
&(z^q-z^{q-1}-1)\displaystyle\sum_{n=1}^{q^2-1}\, \Big[\mathcal{D}_{n,k} - \Big(\frac{k(n-1)+2}{2^n}\Big)\Big] z^n \cr &=(1+z^{q-1}-z^q)\,\displaystyle\sum_{i=1}^{q^2-1}\,z^i\, - \Big(\,z^{2(q-1)}+ \,\displaystyle\sum_{m=1}^{q-1}\,(z-1)^{q-1-m}\,z^{2m}\,\Big(\frac{1}{4}\Big)^m \Big) \,\Big( \displaystyle\sum_{j=0}^{q^2-q+1}\,b_jz^j \Big).
\end{split}
\end{equation}

Let $d_{n,k}=\mathcal{D}_{n,k} - \Big(\dfrac{k(n-1)+2}{2^n}\Big)$ and the right hand side of \eqref{E4.4} be $\displaystyle\sum_{i=1}^{q^2+q-1}\,c_iz^i$.

Then we have

\begin{equation}\label{E4.5}
(z^q-z^{q-1}-1)\,\displaystyle\sum_{n=1}^{q^2-1}\, d_{n,k} z^n = \displaystyle\sum_{i=1}^{q^2+q-1}\,c_i\,z^i.
\end{equation}

\begin{lem}\label{L4.1} (See \cite{Hong-Qin-Zhao-FFA-2016-2}) Let 

\begin{equation}\label{E44.5}
(z^q-z^{q-1}-1)\,\displaystyle\sum_{n=1}^{q^2-1}\, d_{n} z^n = \displaystyle\sum_{i=1}^{q^2+q-1}\,c_i\,z^i.
\end{equation}

By comparing the coefficient of $z^i$ on both sides of \eqref{E44.5}, we have the following.\\

\noindent $d_j=-c_j$ if $1\leq j\leq q-1$;\\
$d_q=c_1-c_q$;\\
$d_{lq+j}=d_{(l-1)q+j} - d_{(l-1)q+j+1} - c_{lq+j}$ if $1\leq l\leq q-2$ and $1\leq j\leq q-1$;\\
$d_{lq}=d_{(l-1)q} - d_{(l-1)q+1} - c_{lq}$ if $2\leq l\leq q-2$;\\
$d_{q^2-q+j}=\displaystyle\sum_{i=j}^{q-1}\,c_{q^2+i}$ if $0\leq j\leq q-1$.
\end{lem}

The following theorem is an immediate consequence of \eqref{E4.5},  Lemma~\ref{L4.1}, and the fact that $d_{n,k}:=\displaystyle\sum_{a\in \f_q} D_{n,k}(1,a)- \Big(\dfrac{k(n-1)+2}{2^n}\Big)$.

\begin{thm}\label{T4.2}
Let $c_j$ be defined as in \eqref{E4.5} for $1\leq j\leq q^2+q-1$. Then we have the following. \\

$\displaystyle\sum_{a\in \f_q} D_{j,k}(1,a)= -c_j + \frac{k(j-1)+2}{2^{j}}\,\, \textnormal{if}\,\, 1\leq j\leq q-1;$\\

$\displaystyle\sum_{a\in \f_q} D_{q,k}(1,a)= c_1-c_q+\frac{2-k}{2^q};$

\[
\begin{split}
\displaystyle\sum_{a\in \f_q} D_{lq+j,k}&=\displaystyle\sum_{a\in \f_q} D_{(l-1)q+j,k} - \displaystyle\sum_{a\in \f_q} D_{(l-1)q+j+1,k}-c_{lq+j} \cr &+\displaystyle\frac{(kj+2)(1-2^q+2^{q-1})+k(2^q-1)}{2^{lq+j}}\, \textnormal{if}\, 1\leq l\leq q-2 \,\textnormal{and} \,1\leq j\leq q-1;
\end{split}
\]

$\displaystyle\sum_{a\in \f_q} D_{lq,k}=\displaystyle\sum_{a\in \f_q} D_{(l-1)q,k} - \displaystyle\sum_{a\in \f_q} D_{(l-1)q+1,k} - c_{lq}+\displaystyle\frac{(k-2)(2^q-1)+2^q}{2^{lq}}$ if $2\leq l\leq q-2$;\\

$\displaystyle\sum_{a\in \f_q} D_{q^2-q+j,k}=\displaystyle\sum_{i=j}^{q-1}\,c_{q^2+i}+\frac{k(j-1)+2}{2^{q^2-q+j}}$ if $0\leq j\leq q-1$.

\end{thm}

\begin{rmk}
Let $k=1$ in Theorem~\ref{T4.2}. Then we have \cite[Theorem~4.1]{Hong-Qin-Zhao-FFA-2016-2}.
\end{rmk}

\begin{rmk}
Let $k=2$ in Theorem~\ref{T4.2}. Then we have \cite[Theorem~4.2]{Fernando-2016}.
\end{rmk}


\end{document}